\documentclass[a4paper]{amsart}
\usepackage{graphicx}
\usepackage{amssymb}
\usepackage{amsmath}
\usepackage{amsthm}
\usepackage{amscd}
\usepackage[all,2cell]{xy}
\UseAllTwocells \SilentMatrices
\newtheorem{thm}{Theorem}[section]

\newtheorem{cor}[thm]{Corollary}
\newtheorem{lem}[thm]{Lemma}
\newtheorem{exm}[thm]{Example}

\newtheorem{prop}[thm]{Proposition}
\theoremstyle{definition}
\newtheorem{defn}[thm]{Definition}
\theoremstyle{remark}
\newtheorem{rem}[thm]{\bf Remark}
\numberwithin{equation}{section}

\begin{document}
\title[The Gorenstein-projective modules over a monomial algebra]
{The Gorenstein-projective modules over a monomial algebra}
\author[Xiao-Wu Chen, Dawei Shen, Guodong Zhou]{Xiao-Wu Chen, Dawei Shen, Guodong Zhou$^*$}

%\thanks{$^*$ The corresponding author}
%\thanks{}
\date{version of \today}

\thanks{E-mail: xwchen@mail.ustc.edu.cn, sdw12345@mail.ustc.edu.cn, gdzhou@math.ecnu.edu.cn}
\thanks{$^*$ the corresponding author}

\maketitle

\dedicatory{}%
\commby{}%
%\begin{center}
%\end{center}

\begin{abstract}
 We classify indecomposable  non-projective Gorenstein-projective modules over a  monomial algebra via the notion of  perfect paths. We apply this classification to a quadratic monomial algebra and describe explicitly the stable category of its Gorenstein-projective modules.
\end{abstract}

\section{Introduction}

Let $A$ be a finite-dimensional algebra over a field. We consider the category of finite-dimensional left $A$-modules. The study of Gorenstein-projective modules goes back to \cite{ABr} under the name ``modules of G-dimension zero". The current terminology is taken from \cite{EJ1}. By the fundamental work \cite{Buc}, the stable category of Gorenstein-projective $A$-modules modulo projective modules is closely related to the singularity category of $A$. Indeed, for a Gorenstein algebra, these two categories are triangle equivalent; also see \cite{Hap2}.

We recall that projective modules are Gorenstein-projective. For a selfinjective algebra, all modules are Gorenstein-projective. Hence for the study of Gorenstein-projective modules, we often consider non-selfinjective algebras. However, there are algebras  which admit no nontrivial Gorenstein-projective modules, that is, any Gorenstein-projective module is actually projective; see \cite{Chen2012}.

There are very few classes of non-selfinjective algebras, for which an explicit classification of indecomposable Gorenstein-projective modules is known. In \cite{Ringel}, such  a classification is obtained for Nakayama algebras; compare \cite{CY}. Using the representation theory of string algebras, there is also such a classification for gentle algebras in \cite{Kal}.

We are interested in the Gorenstein-projective modules over a monomial algebra $A$. It turns out that there is an explicit classification of indecomposable Gorenstein-projective $A$-modules, so that we unify the results in \cite{Ringel} and \cite{Kal} to some extent. We mention that we rely on a fundamental result in \cite{Zim}, which implies in particular that an indecomposable Gorenstein-projective $A$-module is isomorphic to a cyclic module generated by a path. Then the classification pins down to the following question: for which path, the corresponding cyclic module is Gorenstein-projective. The main goal of this work is to answer this question.

Let us describe the content of this paper. In Section 2, we recall basic facts on Gorenstein-projective modules. In Section 3, we introduce the notion of a perfect pair of paths for a monomial algebra $A$. The basic properties of a perfect pair are studied. We introduce the notion of a perfect path in $A$; see Definition \ref{defn:perfectpaths}. We prove the main classification result in Section 4, which claims that there is a bijection between the set of all perfect paths and the set of isoclasses of indecomposable non-projective Gorenstein-projective $A$-modules; see Theorem \ref{thm:main}. We describe in Proposition \ref{prop:answer} a nonzero path $p$ for which the cyclic $A$-module $Ap$ is Gorenstein-projective, thus answer the question above. As an application, we show that for a connected truncated  quiver algebra $A$ with the underlying quiver having no sources or sinks, either $A$ is seflinjective, or any Gorenstein-projective $A$-module is projective; see Example \ref{exm:truncated2}.

We specialize Theorem \ref{thm:main} to a quadratic monomial algebra $A$ in Section 5, in which case any perfect path is an arrow. We introduce the notion of a relation quiver for $A$, whose vertices are given by arrows in $A$ and whose arrows are given by relations in $A$; see Definition \ref{defn:re-quiver}. We prove that an arrow in $A$ is perfect if and only if the corresponding vertex in the relation quiver belongs to a connected component which is a basic cycle; see Lemma \ref{lem:re-quiver}. Using the relation quiver, we obtain a characterization result  on when a quadratic monomial algebra is Gorenstein, which includes the well-known result in \cite{GR} that a gentle algebra is Gorenstein; see Proposition \ref{prop:charac}. We describe explicitly the stable category of Gorenstein-projective $A$-modules, which is proved to be a semisimple abelian category; see Theorem \ref{thm:stable}. This theorem generalizes the main result in \cite{Kal}. More generally, we characterize in Proposition \ref{prop:overlap} for a monomial algebra $A$ when the stable category of Gorenstein-projective $A$-module is semisimple.

In Section 6, we study a Nakayama monomial algebra $A$, where the quiver of $A$ is a basic cycle. Following the idea in \cite{Ringel}, we describe explicitly in Proposition \ref{prop:Nak}  all the perfect paths for $A$. As a consequence, we recover a key characterization result for indecomposable Gorenstein-projective modules over a Nakayama algebra in \cite{Ringel}.

The standard reference on the representation theory of finite-dimensional algebras is \cite{ARS}.

\section{The Gorenstein-projective modules}

In this section, we recall some basic facts about Gorenstein-projective modules over a finite-dimensional algebra.

Let $A$ be a finite-dimensional algebra over a field $k$. We  consider the category  $A\mbox{-mod}$ of finite-dimensional left $A$-modules, and  denote by $A\mbox{-proj}$ the full subcategory of  $A\mbox{-mod}$ consisting of projective $A$-modules. We will identify right $A$-modules as left $A^{\rm op}$-modules, where $A^{\rm op}$ is the opposite algebra of $A$.

 For two left $A$-modules $X$ and $Y$, we denote by ${\rm Hom}_A(X, Y)$ the space   of module homomorphisms from $X$ to $Y$, and by $P(X, Y)$ the subspace formed by those homomorphisms factoring through a projective module. Write ${\rm \underline{Hom}}_A(X, Y)={\rm Hom}_A(X, Y)/P(X, Y)$ to be the quotient space, which is the Hom-space in the stable category $A\mbox{-\underline{mod}}$. Indeed, the stable category $A\mbox{-\underline{mod}}$ is defined as follows: the objects are (finite-dimensional) left $A$-modules, and the Hom-space for two objects $X$ and $Y$ are defined to be ${\rm \underline{Hom}}_A(X, Y)$, where the composition of morphisms is induced by the composition of module homomorphisms.

 Let $M$ be a left $A$-module. Then $M^*={\rm Hom}_A(M, A)$ is a right $A$-module. Recall that an $A$-module $M$ is \emph{Gorenstein-projective} provided that there is an acyclic complex $P^\bullet$ of projective $A$-modules such that the Hom-complex $(P^\bullet)^*={\rm Hom}_A(P^\bullet, A)$ is still acyclic and that $M$ is isomorphic to a certain cocycle $Z^i(P^\bullet)$ of $P^\bullet$. We denote by $A\mbox{-Gproj}$ the full subcategory of $A\mbox{-mod}$ formed by Gorenstein-projective $A$-modules. We observe that $A\mbox{-proj}\subseteq A\mbox{-Gproj}$.  We recall that the full subcategory $A\mbox{-Gproj}\subseteq A\mbox{-mod}$ is closed under direct summands, kernels of epimorphisms and extensions; compare \cite[(3.11)]{ABr}.

The  Gorenstein-projective modules are also called Cohen-Macaulay modules in the literature. Following \cite{Bel2}, an algebra $A$ is \emph{CM-finite} provided that up to isomorphism there are only finitely many indecomposable Gorenstein-projective $A$-modules. As an extreme case, we say that the algebra $A$  is \emph{CM-free} \cite{Chen2012} provided that $A\mbox{-proj}=A\mbox{-Gproj}$.

Let $M$ be a left $A$-module. Recall that its \emph{syzygy} $\Omega(M)=\Omega^1(M)$ is defined to be the kernel of
its projective cover $P\rightarrow M$. Then we have the $d$-th syzygy $\Omega^d(M)$ of $M$ defined inductively by $\Omega^d(M)=\Omega(\Omega^{d-1}M)$ for $d\geq 2$. Set $\Omega^0(M)=M$. We observe that for a Gorenstein-projective module $M$, all its syzygies  $\Omega^d(M)$ are Gorenstein-projective.

Since $A\mbox{-Gproj}\subseteq A\mbox{-mod}$  is closed under
extensions, it becomes naturally an exact category in the sense of
Quillen \cite{Qui73}. Moreover, it is a \emph{Frobenius category},
that is, it has enough (relatively) projective and enough
(relatively) injective objects, and the class of projective objects
coincides with the class of injective objects. In fact, the class of
the projective-injective objects in $A\mbox{-Gproj}$ equals
$A\mbox{-proj}$. In particular, we have ${\rm Ext}_A^i(M, A)=0$ for any Gorenstein-projective $A$-module $M$ and each $i\geq $1. For details, we refer to \cite[Proposition 3.8(i)]{Bel2}.

We denote by $A\mbox{-\underline{Gproj}}$ the full subcategory of $A\mbox{-\underline{mod}}$ consisting of Gorenstein-projective $A$-modules. Then the assignment
$M\mapsto \Omega(M)$ induces an auto-equivalence $\Omega\colon
A\mbox{-\underline{Gproj}}\rightarrow A\mbox{-\underline{Gproj}}$. Moreover, the stable category $A\mbox{-\underline{Gproj}}$ becomes a triangulated category such that  the translation functor is given by a quasi-inverse of $\Omega$, and that the triangles are induced by short exact sequences in $A\mbox{-Gproj}$. These are consequences of a general result in  \cite[Chapter I.2]{Hap}.

We observe that the stable category $A\mbox{-\underline{Gproj}}$ is Krull-Schmidt. We denote by ${\rm ind}\; A\mbox{-\underline{Gproj}}$ the set of isoclasses of indecomposable objects inside. There is a natural identification between ${\rm ind}\; A\mbox{-\underline{Gproj}}$ and the set of  isoclasses of indecomposable non-projective Gorenstein-projective $A$-modules.

The following facts are well known.

\begin{lem}\label{lem:general}
Let $M$ be a Gorenstein-projective $A$-module which is
indecomposable and  non-projective. Then  the following three statements hold.
\begin{enumerate}
\item The syzygy  $\Omega(M)$ is also an indecomposable non-projective Gorenstein-projective $A$-module.
\item  There exists an indecomposable $A$-module $N$ which is non-projective and Gorenstein-projective such that $M\simeq \Omega(N)$.
\item If the algebra $A$ is CM-finite with precisely $d$ indecomposable non-projective Gorenstein-projective modules up to isomorphism, then we have an isomorphism $M\simeq \Omega^{d!}(M)$. Here, $d!$ denotes  the factorial of $d$.  \hfill $\square$
\end{enumerate}
\end{lem}

\begin{proof}
We observe that the auto-equivalence $\Omega\colon
A\mbox{-\underline{Gproj}}\rightarrow A\mbox{-\underline{Gproj}}$ induces a permutation on the set of isoclasses of indecomposable non-projective Gorenstein-projective $A$-modules.
Then all the statements follow immediately. For a detailed proof of (1), we refer to \cite[Lemma 2.2]{Chen2012}.
\end{proof}

Let $d\geq 0$. We recall from \cite{Buc,Hap2} that an algebra $A$ is \emph{$d$-Gorenstein} provided that the injective dimension of the regular module $A$ on both sides is at most $d$. By a \emph{Gorenstein algebra} we mean a $d$-Gorenstein algebra for some $d\geq 0$. We observe that $0$-Gorenstein algebras coincide with selfinjective algebras.

The following result is due to \cite{Hoshi}.

\begin{lem}\label{lem:Gorenstein}
Let $A$ be a finite-dimensional algebra and  $d\geq 0$. Then the  algebra $A$ is $d$-Gorenstein if and only if for each $A$-module $M$, the module $\Omega^d(M)$ is Gorenstein-projective. \hfill $\square$
\end{lem}

For an element $a$ in $A$, we consider the left ideal $Aa$ and the right ideal $aA$ generated by $a$.  We have the following well-defined monomorphism of right $A$-modules
\begin{align}\label{equ:0}
\theta_a\colon aA\longrightarrow (Aa)^*={\rm Hom}_A(Aa, A),
\end{align}
which is defined by $\theta_a(ax)(y)=yx$ for $ax\in aA$  and $y\in Aa$. Dually, we have the following monomorphism of left $A$-modules
\begin{align}\label{equ:00}
\theta'_a\colon Aa\longrightarrow (aA)^*={\rm Hom}_{A^{\rm op}}(aA, A),
\end{align}
which is defined by $\theta'_a(xa)(y)=xy$ for $xa\in Aa$ and $y\in aA$. If $e$ is an idempotent in $A$,  both $\theta_{e}$ and $\theta'_{e}$ are isomorphisms; see \cite[Proposition I.4.9]{ARS}.

The following fact will be used later.

\begin{lem}\label{lem:Hom}
Let $a$ be an element in $A$ satisfying that $\theta_a$ is an isomorphism, and let $b\in A$. Then the isomorphism $\theta_a$ induces a $k$-linear isomorphism
\begin{align}
\frac{aA\cap Ab}{aAb}\stackrel{\sim}\longrightarrow \underline{\rm Hom}_A(Aa, Ab).
\end{align}
\end{lem}

\begin{proof}
For a left ideal $K\subseteq A$, we identify ${\rm Hom}_A(Aa, K)$ with the subspace of ${\rm Hom}_A(Aa, A)$, that consists of homomorphisms with images in $K$. Therefore, the isomorphism $\theta_a$ induces an isomorphism $aA\cap K \stackrel{\sim}\rightarrow {\rm Hom}_A(Aa, K)$. In particular, we have an isomorphism $aA\cap Ab\stackrel{\sim}\rightarrow {\rm Hom}_A(Aa, Ab)$.

Consider the surjective homomorphism $\pi \colon A\rightarrow Ab$ satisfying $\pi(1)=b$. Recall that $P(Aa, Ab)$ denotes the subspace of ${\rm Hom}_A(Aa, Ab)$ consisting of homomorphisms factoring through a projective module. Then $P(Aa, Ab)$ equals the image of ${\rm Hom}_A(Aa, \pi)$. We identify ${\rm Hom}_A(Aa, A)$ with $aA$, and  ${\rm Hom}_A(Aa, Ab)$ with $aA\cap Ab$. We compute that the image of ${\rm Hom}_A(Aa, \pi)$ is identified with $aAb$. Then the required isomorphism follows immediately.
\end{proof}

\section{Monomial algebras and perfect pairs}

In this section, we recall some basic notions and results about a monomial algebra. We introduce the notions of a perfect pair and of a perfect path. Some basic properties of a perfect pair are studied.

Let $Q$  be a finite quiver. We recall that a finite quiver $Q=(Q_0, Q_1; s, t)$ consists of a finite set $Q_0$ of vertices, a finite set $Q_1$ of arrows and two maps $s,t\colon Q_1\rightarrow Q_0$ which assign to each arrow $\alpha$ its starting vertex $s(\alpha)$ and its terminating vertex $t(\alpha)$.

A path $p$ of length $n$ in $Q$ is a sequence $p=\alpha_n\cdots \alpha_2\alpha_1$ of arrows such that $s(\alpha_i)=t(\alpha_{i-1})$ for $2\leq i\leq n$; moreover, we define its starting vertex $s(p)=s(\alpha_1)$ and its terminating vertex $t(p)=t(\alpha_l)$.  We observe that a path of length one is just an arrow. For each vertex $i$, we associate a trivial path $e_i$ of length zero, and set $s(e_i)=i=t(e_i)$. A path of length at least one is said to be \emph{nontrivial}. A nontrivial path is called an \emph{oriented cycle} if its starting vertex equals its terminating vertex.

For two paths $p$ and $q$ with $s(p)=t(q)$, we write $pq$ for their concatenation. By convention, we have $p=pe_{s(p)}=e_{t(p)}p$. For two paths $p$ and $q$ in $Q$, we say that $q$ is a \emph{sub-path} of $p$ provided that $p=p''qp'$ for some paths $p''$ and $p'$. The sub-path $q$ is \emph{proper} if further $p\neq q$.

Let  $S$ be a set of paths in $Q$. A path $p$ in $S$ is \emph{left-minimal in $S$} provided that there is no path $q$ such that $q\in S$ and $p=qp'$ for some nontrivial path $p'$. Dually, one defines a \emph{right-minimal path in $S$}. A path $p$ in $S$ is  \emph{minimal in $S$} provided that there is no proper sub-path $q$ of $p$ inside $S$.

Let $k$ be a field. The path algebra $kQ$ of a finite quiver $Q$ is defined as follows. As a $k$-vector space, it has a basis given by all the paths in $Q$. For two paths $p$ and $q$, their multiplication is given by the concatenation $pq$ if $s(p)=t(q)$, and it is zero, otherwise.  The unit of $kQ$ equals  $\sum_{i\in Q_0}e_i$. Denote by $J$ the two-sided ideal of $kQ$ generated by arrows. Then $J^d$ is spanned by all the paths of length at least $d$ for each $d\geq 2$. A two-sided ideal $I$ of $kQ$ is \emph{admissible} if $J^d\subseteq I\subseteq J^2$ for some $d\geq 2$. In this case, the quotient algebra $A=kQ/I$ is finite-dimensional.

We recall that an admissible ideal $I$ of $kQ$ is \emph{monomial} provided that it is generated by some paths of length at least two. In this case, the quotient algebra $A=kQ/I$ is called a \emph{monomial algebra}.

Let $A=kQ/I$ be a monomial algebra as above. We denote by $\mathbf{F}$ the set formed by all the minimal paths among the paths in $I$; it is a finite set. Indeed, the set $\mathbf{F}$ generates $I$ as a two-sided ideal. Moreover, any set consisting of paths that generates $I$ necessarily contains $\mathbf{F}$.

We make the following convention as in \cite{Zim}. A path $p$ is said to be a  \emph{nonzero path  in $A$} provided that $p$ does not belong to $I$, or equivalently, $p$ does not contain a sub-path in $\mathbf{F}$. For a nonzero path $p$, we abuse $p$ with its canonical image $p+I$ in $A$. On the other hand, for a path $p$ in $I$, we write $p=0$ in $A$.  We observe that the set of nonzero paths forms a $k$-basis of $A$.

For a nonzero path $p$,  we consider the left ideal $Ap$ and the right ideal $pA$. We observe that $Ap$ has a basis given by all nonzero paths $q$ such that $q=q'p$ for some path $q'$. Similarly, $pA$ has a basis given by all nonzero paths $\gamma$ such that $\gamma=p\gamma'$ for some path $\gamma'$. If $p=e_i$ is trivial, then $Ae_i$ and $e_iA$ are indecomposable projective left and right $A$-modules, respectively.

For a nonzero nontrivial path $p$, we define $L(p)$ to be the set of right-minimal paths in the set  $\{\mbox{nonzero paths } q\; |\; s(q)=t(p) \mbox{ and } qp=0\}$ and $R(p)$ the set of left-minimal paths in the set $\{\mbox{nonzero paths } q\; |\; t(q)=s(p) \mbox{ and } pq=0\}$.

The following well-known fact is straightforward; compare the first paragraph in \cite[p.162]{Zim}.

\begin{lem}\label{lem:trivial}
Let $p$ be a nonzero nontrivial path in $A$. Then we have the following exact sequence of left $A$-modules
\begin{align}\label{equ:1}
0\longrightarrow \bigoplus_{q\in L(p)} Aq\stackrel{\rm inc}\longrightarrow Ae_{t(p)}\stackrel{\pi_p}\longrightarrow Ap\longrightarrow 0,
\end{align}
where ``${\rm inc}$" is the inclusion map and $\pi_p$ is the projective cover of $Ap$ with $\pi_p(e_{t(p)})=p$. Similarly, we have the following exact sequence of right $A$-modules
\begin{align}\label{equ:2}
0\longrightarrow \bigoplus_{q\in R(p)} qA\stackrel{\rm inc}\longrightarrow e_{s(p)}A \stackrel{\pi'_p}\longrightarrow pA\longrightarrow 0,
\end{align}
where $\pi'_p$ is the projective cover of $pA$ with $\pi'_p(e_{s(p)})=p$. \hfill $\square$
\end{lem}

We will rely on the following fundamental result contained in \cite[Theorem I]{Zim}.

\begin{lem}\label{lem:basic}
Let $M$ be a left $A$-module which fits into an exact sequence $0\rightarrow M\rightarrow P\rightarrow Q$ of $A$-modules with $P, Q$ projective. Then $M$ is isomorphic to a direct sum $\bigoplus Ap^{(\Lambda(p))}$, where $p$ runs over all the nonzero paths in $A$ and each $\Lambda(p)$ is some index set. \hfill $\square$
\end{lem}

The main notion we need is as follows.

\begin{defn}\label{defn:perfectpair}
Let $A=kQ/I$ be a monomial algebra as above. We call a pair $(p, q)$ of nonzero paths in $A$ \emph{perfect} provided that the following conditions are satisfied:
\begin{enumerate}
\item[(P1)] both of the nonzero paths $p,q$  are nontrivial satisfying $s(p)=t(q)$ and $pq=0$ in $A$;
\item[(P2)] if $pq'=0$ for a nonzero path $q'$ with $t(q')=s(p)$, then $q'=qq''$ for some path $q''$; in other words, $R(p)=\{q\}$;
\item[(P3)] if $p'q=0$  for a nonzero path $p'$ with $s(p')=t(q)$, then $p'=p''p$ form some path $p''$; in other words, $L(q)=\{p\}$.\hfill $\square$
\end{enumerate}
\end{defn}

Let $(p, q)$ be a perfect pair. Applying (P3) to (\ref{equ:1}),  we have the following exact sequences of left $A$-modules
\begin{align}\label{equ:3}
0\longrightarrow  Ap\stackrel{\rm inc}\longrightarrow Ae_{t(q)}\stackrel{\pi_q}\longrightarrow Aq\longrightarrow 0.
\end{align}
In particular, we have that $\Omega(Aq)\simeq Ap$.

The following result seems to be convenient for computing  perfect pairs.

\begin{lem}\label{lem:char}
Recall that $\mathbf{F}$ denotes the finite set of minimal paths contained in $I$. Let $p$ and $q$ be nonzero nontrivial paths in $A$ satisfying $s(p)=t(q)$. Then the pair $(p, q)$ is perfect if and only if the following three conditions are satisfied:
\begin{enumerate}
\item[${\rm (P'1)}$] the concatenation $pq$ lies in $\mathbf{F}$;
\item[${\rm (P'2)}$] if $q'$ is a nonzero path in $A$ satisfying $t(q')=s(p)$ and $pq'=\gamma \delta$ for a path $\gamma$  and some path $\delta\in \mathbf{F}$, then $q'=qx$ for some path $x$;
\item[${\rm (P'3)}$]  if $p'$ is a nonzero path satisfying $s(p')=t(q)$ and $p'q=\delta \gamma$ for a path $\gamma$ and some path $\delta\in \mathbf{F}$, then $p'=yp$ for some path $y$.
\end{enumerate}
\end{lem}

\begin{proof}
For the ``only if" part, we assume that $(p,q)$ is a perfect pair. By (P1) we have $pq=\gamma_2\delta \gamma_1$ with $\delta\in \mathbf{F}$ and some paths $\gamma_1$ and $\gamma_2$. We claim that $\gamma_1=e_{s(q)}$. Otherwise, $q=q'\gamma_1$ for a proper sub-path $q'$, and  thus $pq'=\gamma_2\delta$ which equals zero in $A$. This contradicts (P2). Similarly, we have $\gamma_2=e_{t(p)}$. Then we infer ${\rm (P'1)}$. The conditions ${\rm (P'2)}$ and ${\rm (P'3)}$ follow from (P2) and (P3) immediately.

For the ``if" part, we observe that the condition (P1) is immediate. For (P2), assume that $pq'=0$ in $A$, that is, $pq'=\gamma\delta\gamma_1$ with $\delta\in \mathbf{F}$ and some paths $\gamma$ and $\gamma_1$. We have $q'=q''\gamma_1$ for some path $q''$, since the path $p$ is nonzero. Then we have $pq''=\gamma\delta$. By  ${\rm (P'2)}$ we infer that $q''=qx$ and thus $q'=q(x\gamma_1)$. This proves (P2). Similarly, we have (P3).
\end{proof}

We study some basic properties of a perfect pair in the following lemmas.

\begin{lem}\label{lem:unique}
Let $(p, q)$ and $(p', q')$ be two perfect pairs. Then the following statements are equivalent:
\begin{enumerate}
\item $(p, q)=(p', q')$;
\item there is an isomorphism $Aq\simeq Aq'$ of left $A$-modules;
\item there is an isomorphism $pA\simeq p'A$ of right $A$-modules.
\end{enumerate}
\end{lem}

\begin{proof}
We only prove ``(2)$\Rightarrow$(1)". Assume that $\phi\colon Aq\rightarrow Aq'$ is an isomorphism. Consider the projective covers $\pi_q\colon Ae_{t(q)}\rightarrow Aq$ and $\pi_{q'}\colon Ae_{t(q')}\rightarrow Aq'$. Then there is an isomorphism $\psi\colon Ae_{t(q)}\rightarrow Ae_{t(q')}$ such that $\pi_{q'}\circ \psi=\phi\circ \pi_{q}$. In particular, we have $t(q)=t(q')$.

Assume that $\psi(e_{t(q)})=\lambda e_{t(q)} +\sum \lambda(\gamma)\gamma$, where $\lambda$ and $\lambda(\gamma)$ are in $k$ and $\gamma$ runs over all the nonzero nontrivial paths that start at $t(q)$. Since $\psi$ is an isomorphism, we infer that $\lambda \neq 0$. We observe that $\psi(p)=\lambda p+\sum \lambda(\gamma)p\gamma$. Recall that  $\pi_q(p)=pq=0$. By $\pi_{q'}\circ \psi=\phi\circ \pi_{q}$, we have $\psi(p)q'=0$. We then infer that $pq'=0$ and thus $p=\delta p'$ for some path $\delta$. Similarly, we have $p'=\delta' p$. We conclude that $p=p'$. Since $R(p)=\{q\}$ and $R(p')=\{q'\}$, we infer that $q=q'$. Then we are done.
\end{proof}

We recall the homomorphisms in (\ref{equ:0}) and (\ref{equ:00}).

\begin{lem}\label{lem:iso}
Let $(p,q)$ be a perfect pair. Then both $\theta_q$ and $\theta'_p$ are isomorphisms.
\end{lem}

\begin{proof}
We have mentioned that the map $\theta_q$ is a monomorphism.

To show that it is epic, take a homomorphism $f\colon Aq\rightarrow A$ of left $A$-modules. Since $q=e_{t(q)}q$, we infer that $f(q)$ belongs to $e_{t(q)}A$. We assume that $f(q)=\sum \lambda(\gamma) \gamma$, where each $\lambda(\gamma)$ is in $k$ and $\gamma$ runs over all nonzero paths terminating at $t(q)$. By $pq=0$, we deduce that $p\gamma=0$ for those $\gamma$ with $\lambda(\gamma)\neq 0$. By (P2) each $\gamma$ lies in $qA$. Therefore, we infer that $f(q)$ lies in $qA$, and thus $\theta_q(f(q))=f$. Then we infer that $\theta_q$ is an isomorphism. Dually, one proves that $\theta'_p$ is also an isomorphism.
\end{proof}

The following notion plays a central role in this paper. 

\begin{defn}\label{defn:perfectpaths}
Let $A=kQ/I$ be a monomial algebra. We call a nonzero path $p$ in $A$ a \emph{perfect path}, provided that there exists a sequence $$p=p_1, p_2, \cdots, p_n, p_{n+1}=p$$
 of nonzero paths such that $(p_i, p_{i+1})$ are perfect pairs for all $1\leq i\leq n$. If the given nonzero paths $p_i$ are pairwise distinct, we refer to the sequence $p=p_1, p_2, \cdots, p_n, p_{n+1}=p$ as a \emph{relation-cycle} for $p$. \hfill $\square$
\end{defn}

We mention that by Definition \ref{defn:perfectpair}(P2) a perfect path has a unique relation-cycle and the  notion of a relation-cycle is analogous to  cyclic paths considered in \cite[Section 2]{Kal} for gentle algebras.

Let $n\geq 1$. By a \emph{basic ($n$-)cycle}, we mean a quiver consisting of $n$ vertices and $n$ arrows which form an oriented cycle.

\begin{exm}\label{exm:truncated}
{\rm Let $Q$ be a connected quiver and let $d\geq 2$. Recall that $J$ denotes the two-sided  ideal of $kQ$ generated by arrows. The monomial algebra $A=kQ/J^d$ is called a \emph{truncated quiver algebra}. When the quiver $Q$ has no sources or sinks, we claim that $A$ has a perfect path if and only if the quiver $Q$ is a basic cycle.

Indeed, let $(p, q)$ be a perfect pair. Assume that $p=\alpha_n\cdots \alpha_2\alpha_1$ and $q=\beta_m\cdots \beta_2\beta_1$. We observe that each $\alpha_i$ is the unique arrow  starting at $s(\alpha_i)$. Otherwise, $L(q)$ has at least two elements, contradicting to (P3). Here, we use the fact that $Q$ has no sinks.  Similarly, we observe that each $\beta_j$ is the unique arrow terminating at $t(\beta_j)$, using the fact that $Q$ has no souces. Let $p$ be a perfect path with a relation-cycle  $p=p_1, p_2, \cdots, p_n, p_{n+1}=p$. We apply the above observations to the perfect pairs $(p_i, p_{i+1})$. Then we obtain that the quiver $Q$ is a basic cycle.

On the other hand, if $Q$ is a basic cycle, then any nonzero nontrivial path $p$ of length strictly less than $d-1$ is perfect. }
\end{exm}

\section{The Gorenstein-projective modules over a monomial algebra}

In this section, we parameterize the isoclasses of indecomposable non-projective Gorenstein-projective modules over a monomial algebra by perfect paths.

Recall that for a finite-dimensional algebra $A$, the set ${\rm ind}\; A\mbox{-\underline{\rm Gproj}}$ of isoclasses of indecomposable objects in the stable category  $A\mbox{-\underline{\rm Gproj}}$  is identified with the set of isoclasses of indecomposable non-projective Gorenstein-projective $A$-modules.

 The main result of this paper is as follows.

\begin{thm}\label{thm:main}
Let $A$ be a monomial algebra. Then there is a bijection
$$\{\mbox{perfect paths in }A\}\stackrel{1:1} \longleftrightarrow {\rm ind}\; A\mbox{-\underline{\rm Gproj}}$$
sending a perfect path $p$ to the $A$-module $Ap$.
\end{thm}

\begin{proof} The well-definedness of the map is due to Proposition \ref{prop:bijection}(4). The surjectivity follows from  Proposition \ref{prop:bijection}(3). The injectivity follows from Lemma \ref{lem:unique}.
\end{proof}

\begin{rem}
We mention that by Theorem \ref{thm:main} considerable information on the stable category $A\mbox{-\underline{\rm Gproj}}$ is already obtained. For example, the syzygy functor $\Omega$ on indecomposable objects is computed by (\ref{equ:3}), and the Hom-spaces between indecomposable objects are computed by Lemma \ref{lem:Hom}.

In general, we do not have a complete description for $A\mbox{-\underline{\rm Gproj}}$  as a triangulated category. We do have such a description in the quadratic monomial case, and slightly more generally, in the case that there is no overlap in $A$; see Theorem \ref{thm:stable} and Proposition \ref{prop:overlap}.
\end{rem}

The map $\theta_p$ in the following proposition is introduced in (\ref{equ:0}).

\begin{prop}\label{prop:bijection}
Let $A=kQ/I$ be a monomial algebra. Then the following statements hold.
\begin{enumerate}
\item For a nonzero nontrivial path $p$ satisfying that the morphism $\theta_p$ is an isomorphism, if the $A$-module $Ap$ is non-projective Gorenstein-projective, then the path $p$ is perfect.
\item For a nonzero nontrivial path $p$, if the $A$-module $Ap$ is non-projective Gorenstein-projective, then there is a unique perfect path $q$ such that  $L(p)=\{q\}$.
\item Let $M$ be an indecomposable non-projective Gorenstein-projective $A$-module. Then there exists a perfect path $q$ satisfying that  $M\simeq Aq$.
\item For a perfect path $p$, the $A$-module $Ap$ is non-projective Gorenstein-projective.

\end{enumerate}
\end{prop}

\begin{proof}
We  observe that each indecomposable non-projective  Gorenstein-projective  $A$-module $X$ is of the form $A\gamma$ for some nonzero nontrivial path $\gamma$. Indeed, there exists an exact sequence $0\rightarrow X\rightarrow P\rightarrow Q$ of $A$-modules with $P, Q$ projective. Then we are done by Lemma \ref{lem:basic}. In particular, this observation implies that a monomial algebra $A$ is CM-finite.
\vskip 3pt

(1)   By (\ref{equ:1}) we have $\Omega(Ap)=\bigoplus_{q\in L(p)}Aq$, which is indecomposable and non-projective by Lemma \ref{lem:general}(1). We infer that $L(p)=\{q\}$ for some nonzero nontrivial path $q$. Consider the exact sequence (\ref{equ:1}) for $p$
$$\eta\colon 0\longrightarrow Aq \stackrel{\rm inc}\longrightarrow Ae_{t(p)}\stackrel{\pi_p}\longrightarrow Ap\longrightarrow 0.$$
Recall that the $A$-module $Ap$ is Gorenstein-projective, in particular, ${\rm Ext}_A^1(Ap, A)=0$. Therefore, the lower row of the following commutative diagram is exact.
\[\xymatrix{
\epsilon\colon & 0\ar[r] & pA \ar[d]_{\theta_p}\ar[r]^-{\rm inc} & e_{s(q)}A \ar[d]_{\theta_{e_{s(q)}}} \ar[r]^-{\pi'_q}& qA \ar[d]^-{\theta_q}\ar[r] & 0\\
\eta^*\colon & 0 \ar[r] & (Ap)^* \ar[r]^-{(\pi_p)^*} & (Ae_{t(p)})^* \ar[r]^-{{\rm inc}^*} & (Aq)^*\ar[r] & 0
}\]
Recall that $\theta_{e_{s(q)}}$ is an isomorphism and $\theta_q$ is a monomorphism. However, since ${\rm inc}^*$ is epic, we infer that $\theta_q$ is also epic and thus an isomorphism. We note that the $A$-module $Aq$ is also non-projective Gorenstein-projective; see Lemma \ref{lem:general}(1).

  We claim that $(q, p)$ is a perfect pair. Indeed, we already have (P1) and (P2). It suffices to show that $R(q)=\{p\}$. By assumption, the map $\theta_p$ is an isomorphism. Then the upper sequence $\epsilon$ in the above diagram is also exact. Comparing $\epsilon$ with (\ref{equ:2}), we obtain that $R(q)=\{p\}$.

We have obtained a perfect pair $(q, p)$ and also proved that $\theta_q$ is an isomorphism. We mention that $\Omega(Ap)\simeq Aq$. Set $q_0=p$ and $q_1=q$. We now replace $p$ by $q$ and continue the above argument. Thus we obtain perfect pairs $(q_{m+1}, q_m)$ for all $m\geq 0$ satisfying $\Omega(Aq_m)\simeq Aq_{m+1}$. By Lemma \ref{lem:general}(3), we have  for a sufficiently large $m$, an isomorphism $Aq_m\simeq Aq_0=Ap$. By Lemma \ref{lem:unique} we have $q_m=p$. Thus we have a required sequence $p=q_m, q_{m-1}, \cdots, q_1, q_0=p$, proving that $p$ is perfect.

\vskip 3pt

(2) By the first paragraph in the proof of (1), we obtain that $L(p)=\{q\}$ and that $\theta_q$ is an isomorphism. Since $Aq\simeq \Omega(Ap)$, we infer that $Aq$ is non-projective Gorenstein-projective. Then the path $q$ is perfect by (1), proving (2).

\vskip 3pt

(3) By Lemma \ref{lem:general}(2), there is an indecomposable non-projective Gorenstein-projective $A$-module $N$ such that $M\simeq \Omega(N)$. By the observation above, we may assume that $N=Ap$ for a nonzero nontrivial path $p$. Recall from  (\ref{equ:1}) that $\Omega(N)\simeq \bigoplus_{p'\in L(p)}Ap'$. Then we have $L(p)=\{q\}$ for some nonzero path $q$ and an isomorphism $M\simeq Aq$. The path $q$ is necessarily perfect by (2).

\vskip 3pt

(4)  We take a relation-cycle $p=p_1, p_2, \cdots, p_n, p_{n+1}=p$ for the perfect path $p$. We define $p_{m}=p_j$ if $m=an+j$ for some integer $a$ and $1\leq j\leq n$. Then each pair $(p_m, p_{m+1})$ is perfect. By (\ref{equ:3}) we have an exact sequence of left $A$-modules
$$\eta_m\colon 0\longrightarrow Ap_m\stackrel{\rm inc}\longrightarrow Ae_{t(p_{m+1})}\stackrel{\pi_{p_{m+1}}}\longrightarrow Ap_{m+1}\longrightarrow 0.$$
Gluing all these $\eta_m$'s together,  we obtain an acyclic complex $P^\bullet=\cdots \rightarrow Ae_{t(p_{m})}\rightarrow Ae_{t(p_{m+1})}\rightarrow \cdots$ such that $Ap$ is isomorphic to one of the cocycles. We observe that $Ap=Ap_1$ is non-projective, since $\eta_1$ does not split.

 It remains to prove that the Hom-complex $(P^\bullet)^*={\rm Hom}_A(P^\bullet, A)$ is also acyclic. For this, it suffices to show that for each $m$, the sequence ${\rm Hom}_A(\eta_m, A)$ is exact, or equivalently, the morphism ${\rm inc}^*={\rm Hom}_A({\rm inc}, A)$ is epic. We observe the following commutative diagram
 \[\xymatrix{
  e_{s(p_m)}A\ar[d]_{\theta_{e_{s(p_m)}}} \ar[r]^-{\pi'_{p_m}} & p_m A\ar[d]^{\theta_{p_m}}  \\
 (Ae_{t(p_{m+1})})^* \ar[r]^-{{\rm inc}^*}& (Ap_m)^*,
}\]
where we use the notation in (\ref{equ:2}). Recall that $\theta_{e_{s(p_m)}}$ is an isomorphism. By Lemma \ref{lem:iso} the morphism $\theta_{p_m}$ is an isomorphism. Since $\pi'_{p_m}$ is a projective cover, we infer that  the morphism ${\rm inc}^*$ is epic. We are done with the whole proof.
\end{proof}

The following example shows that the condition that $\theta_p$ is an isomorphism is necessary in the proof of Proposition \ref{prop:bijection}(1).

\begin{exm}
{\rm Let $Q$ be the following quiver.
\[\xymatrix{
1 \ar @/^/[r]^-\alpha & 2\ar @/^/[l]^-\beta & \ar[l]_{\gamma} 3}\]
Let $I$ be the ideal generated by $\beta\alpha$ and $\alpha\beta$, and let $A=kQ/I$. We denote by $S_i$ the simple $A$-module corresponding to the vertex $i$ for $1\leq i\leq 3$.

The corresponding set $\mathbf{F}$ of minimal paths contained in $I$ equals $\{\beta\alpha, \alpha\beta\}$. Then we observe that there exist precisely two perfect pairs, which are $(\beta, \alpha)$ and $(\alpha, \beta)$. Hence, the set of perfect paths equals $\{\alpha, \beta\}$. Then by Theorem \ref{thm:main}, up to isomorphism,  all the indecomposable non-projective Gorenstein-projective $A$-modules are given by $A\alpha$ and $A\beta$. We observe two isomorphisms $A\alpha\simeq S_2$ and $A\beta\simeq S_1$.

Consider the nonzero path $p=\beta\gamma$. Then there is an isomorphism $Ap\simeq A\beta$ of left  $A$-modules. In particular, the $A$-module $Ap$ is  non-projective Gorenstein-projective. However, the path $p$ is not perfect.}
\end{exm}

The following result answers the question asked in the introduction: for which nonzero path $p$ in $A$, the $A$-module $Ap$ is Gorenstein-projective. We observe by (\ref{equ:1}) that $Ap$ is projective if and only if $p$ is trivial or $L(p)$ is empty.

\begin{prop}\label{prop:answer}
Let $A$ be a monomial algebra and let $p$ be a nonzero nontrivial path. Then $Ap$ is non-projective Gorenstein-projective if and only if  $L(p)=\{q\}$ for a perfect path $q$.
\end{prop}

\begin{proof}
The ``only if" part is due to Proposition \ref{prop:bijection}(2). Conversely, let $(q, p')$ be a perfect pair with $p'$ perfect. In particular, $L(p')=\{q\}$ and thus $t(p')=t(p)$. By comparing the exact sequences (\ref{equ:1}) for $p$ and $p'$, we obtain an isomorphism $Ap\simeq Ap'$. Then we are done, since $Ap'$ is non-projective Gorenstein-projective by Proposition \ref{prop:bijection}(4).
\end{proof}

The following example shows that a connected truncated quiver algebra  is either selfinjective or CM-free,  provided that the underlying quiver has no sources or sinks; compare \cite[Theorem 1.1]{Chen2012}.

\begin{exm}\label{exm:truncated2}
{\rm Let $A=kQ/J^d$ be the truncated quiver algebra in Example \ref{exm:truncated} such that $Q$ is connected without sources or sinks. If $Q$ is not a basic cycle, then there is no perfect path. Then by Theorem \ref{thm:main}, the algebra $A$ is CM-free. On the other hand, if $Q$ is a basic cycle, then the algebra $A$ is well known to be selfinjective}.
\end{exm}

\section{The Gorenstein-projective modules over a quadratical monomial algebra}

In this section, we specialize Theorem \ref{thm:main} to a quadratic monomial algebra. We describe explicitly the stable category of Gorenstein-projective modules over a quadratic monomial algebra, which turns out to be a semisimple triangulated category. We characterize for a monomial algebra when its stable category of Gorenstein-projective module is semisimple.

Let $A=kQ/I$ be a monomial algebra. We say that the algebra $A$ is \emph{quadratic monomial} provided that the ideal $I$ is generated by paths of length two, or equivalently, the corresponding set $\mathbf{F}$ consists of certain paths of length two. By Lemma \ref{lem:char}(${\rm P'1}$), for a perfect pair $(p, q)$ in $A$, both $p$ and $q$ are necessarily arrows. In particular, a perfect path is an arrow and its relation-cycle consists entirely of arrows.

Hence, we have the following immediate consequence of Theorem \ref{thm:main}.

\begin{prop}\label{prop:tuilun}
Let $A$ be a quadratic  monomial algebra. Then there is a bijection
$$\{\mbox{perfect arrows in }A\}\stackrel{1:1} \longleftrightarrow {\rm ind}\; A\mbox{-\underline{\rm Gproj}}$$
sending a perfect arrow $\alpha$ to the $A$-module $A\alpha$. \hfill $\square$
\end{prop}

 We will give a more convenient characterization of a perfect arrow. For this end, we introduce the following notion.

\begin{defn}\label{defn:re-quiver}
Let  $A=kQ/I$ be a quadratic monomial algebra. We define its \emph{relation quiver} $\mathcal{R}_A$ as follows: the vertices are given by the arrows in $Q$, and there is an arrow $[\beta\alpha]\colon \alpha\rightarrow \beta$ if $t(\alpha)=s(\beta)$ and $\beta\alpha$ lies in $I$, or equivalently in $\mathbf{F}$.

Let $\mathcal{C}$ be a connected component of $\mathcal{R}_A$. We call $\mathcal{C}$ a \emph{perfect component} (\emph{resp.} an \emph{acyclic component}) if $\mathcal{C}$ is a basic cycle (\emph{resp.} contains no oriented cycles).\hfill $\square$
\end{defn}

We mention that the relation quiver is somehow dual to the Ufnarovskii graph studied in \cite{HS}.

The following lemma implies that an arrow in $Q$ is perfect if and only if the corresponding vertex in the relation quiver of $A$ belongs to a perfect component.

\begin{lem}\label{lem:re-quiver}
Let  $A=kQ/I$ be a quadratic monomial algebra, and let $\alpha$ be an arrow. Then the following statements hold.
\begin{enumerate}
\item We have $L(\alpha)=\{\beta\in Q_1\; |\; s(\beta)=t(\alpha) \mbox{ and } \beta\alpha\in \mathbf{F}\}$, and $R(\alpha)=\{\beta\in Q_1\; |\; t(\beta)=s(\alpha) \mbox{ and }\alpha\beta\in \mathbf{F}\}$.
\item Assume that $\beta$ is an arrow with $t(\beta)=s(\alpha)$. Then the pair $(\alpha, \beta)$ is perfect if and only if there is an arrow $[\alpha\beta]$ from $\beta$ to $\alpha$ in $\mathcal{R}_A$, which is the unique arrow starting at $\beta$ and also the unique arrow terminating at $\alpha$.
    \item  The arrow $\alpha$ is perfect if and only if the corresponding vertex belongs a perfect component of $\mathcal{R}_A$.
    \end{enumerate}
\end{lem}

\begin{proof}
For (1), we observe the following fact: for a nonzero path $p$ with $s(p)=t(\alpha)$, then $p\alpha=0$ if and only if $p=p'\beta$ with $\beta\alpha\in \mathbf{F}$. This fact implies the equation on $L(\alpha)$. Similarly, we have the equation on $R(\alpha)$.

We mention that by (1), the set $L(\alpha)$ consists of all immediate successors of $\alpha$ in $\mathcal{R}_A$, and $R(\alpha)$ consists of all immediate predecessors of $\alpha$. Then (2) follows immediately from the definition of a perfect pair. The statement (3) is an immediate consequence of (2).
\end{proof}

The following result concerns the homological property of the module $A\alpha$. We call a vertex $j$ in a quiver \emph{bounded}, if the lengths of all the paths starting at $j$ are uniformly bounded. In this case, the length is strictly less than the number of vertices in the quiver.

\begin{lem}\label{lem:homological}
Let  $A=kQ/I$ be a quadratic monomial algebra, and let $\alpha$ be an arrow. Then the following statements hold.
\begin{enumerate}
\item The $A$-module $A\alpha$ is non-projective Gorenstein-projective if and only if the corresponding vertex of $\alpha$  belongs a perfect component of $\mathcal{R}_A$.
\item The $A$-module $A\alpha$ has finite projective dimension if and only if $\alpha$ is a bounded vertex in  $\mathcal{R}_A$.
\item If the corresponding vertex of $\alpha$ in $\mathcal{R}_A$ is not bounded and does not belong to a perfect component, then each syzygy module $\Omega^d(A\alpha)$ is not Gorenstein-projective.
    \end{enumerate}
\end{lem}

\begin{proof}
The ``if" part of (1) follows from Lemma \ref{lem:re-quiver}(3) and Proposition \ref{prop:bijection}(4). For the ``only if" part, assume that the  $A$-module $A\alpha$ is non-projective Gorenstein-projective. By Proposition \ref{prop:bijection}(2) there is a perfect arrow $\beta$ such that $L(\alpha)=\{\beta\}$. In particular, there is an arrow from $\alpha$ to $\beta$ in $\mathcal{R}_A$. By Lemma \ref{lem:re-quiver}(3) $\beta$ belongs to a perfect component $\mathcal{C}$ of $\mathcal{R}_A$. It follows that $\alpha$ also belongs to $\mathcal{C}$.

For (2), we observe that by Lemma \ref{lem:re-quiver}(1) and (\ref{equ:1}) there is an isomorphism
\begin{align}\label{equ:syzygy}
\Omega(A\alpha)\stackrel{\sim}\longrightarrow \bigoplus A\beta,
\end{align}
where $\beta$ runs over all the immediate successors of $\alpha$ in $\mathcal{R}_A$. Then (2) follows immediately.

For (3), we assume on the contrary that $\Omega^d(A\alpha)$ is Gorenstein-projective for some $d\geq 1$. We know already by (2) that $\Omega^d(A\alpha)$ is not projective. By iterating the formula  (\ref{equ:syzygy}), we obtain an arrow $\beta$ such that $A\beta$ is non-projective Gorenstein-projective and that there is a path from $\alpha$ to $\beta$ of length $d$ in $\mathcal{R}_A$. However, by (1) $\beta$ belongs to a perfect component $\mathcal{C}$. It follows that $\alpha$ also belongs to $\mathcal{C}$, which is a desired contradiction.
\end{proof}

The following result studies the Gorenstein homological properties of a quadratic monomial algebra. In particular, we obtain a characterization of a quadratic monomial algebra being Gorenstein, which contains the well-known result that a gentle algebra is Gorenstein; see \cite[Theorem 3.4]{GR} and compare Example \ref{exm:quadratic}.

\begin{prop}\label{prop:charac}
Let  $A=kQ/I$ be a quadratic monomial algebra. Denote by $d$ the length of the longest paths in acyclic components of $\mathcal{R}_A$. Then the following statements hold.
\begin{enumerate}
\item The algebra $A$ is Gorenstein if and only if any connected component of its relation quiver $\mathcal{R}_A$ is either perfect or acyclic. In this case, the algebra $A$ is $(d+2)$-Gorenstein.
    \item The algebra $A$ is CM-free if and only if the relation quiver $\mathcal{R}_A$ contains no perfect component.
    \item The algebra $A$ has finite global dimension if and only if any component of the relation quiver $\mathcal{R}_A$ is acyclic.
    \end{enumerate}
\end{prop}

\begin{proof}
The statement (2) is an immediate consequence of Proposition \ref{prop:tuilun} and Lemma \ref{lem:re-quiver}(3), while the final statement is an immediate consequence of (1) and (2). Here, we recall a well-known consequence of Lemma \ref{lem:Gorenstein}: an algebra $A$ has finite global dimension if and only if it is Gorenstein and CM-free.

We now prove (1). Recall from Lemma \ref{lem:Gorenstein} that the algebra $A$ is Gorenstein if and only if there exists a natural number $n$ such that $\Omega^n(M)$ is Gorenstein-projective for any $A$-module $M$.

For ``only if"  part, we assume the contrary. Then there is  an arrow $\alpha$, whose corresponding vertex in $\mathcal{R}_A$ is not bounded and does not belong to a perfect component. By Lemma \ref{lem:homological} (3), we infer that $A$ is not Gorenstein. A contradiction!

For the ``if" part, let $\alpha$ be an arrow. Then the $A$-module $A\alpha$ is either Gorensein-projective or has finite projective dimension at most $d$; see Lemma \ref{lem:homological}(1) and (2). We infer that the $A$-module $\Omega^{d} (A\alpha)$ is Gorenstein-projective. We observe that for a nonzero path $p=\alpha p'$ of length at least two, we have a  module isomorphism $A\alpha \simeq Ap$, sending $x$ to $xp'$. So we conclude that for any nonzero path $p$, the $A$-module $\Omega^{d}(Ap)$ is Gorenstein-projective.

 Let $M$ be any $A$-module. Then by Lemma \ref{lem:basic} we have an isomorphism between $\Omega^2(M)$ and a direct sum of the modules $Ap$ for some nonzero paths $p$. It follows by above that the syzygy module $\Omega^{d+2}(M)$ is Gorenstein-projective. This proves that the algebra $A$ is $(d+2)$-Gorenstein.

\end{proof}

\begin{exm}\label{exm:quadratic}
{\rm Let $A$ be a quadratic monomial algebra. We assume that for each arrow $\alpha$, there exists at most one arrow $\beta$ with $ \alpha\beta\in \mathbf{F}$ and at most one arrow $\gamma$ with $\gamma\alpha\in \mathbf{F}$.  Then the algebra $A$ is Gorenstein. In particular, a gentle algebra  satisfies these conditions. As a consequence, we recover the main part of  \cite[Theorem 3.4]{GR}.

Indeed, the assumption implies that at each vertex in $\mathcal{R}_A$, there is at most one arrow starting and at most one arrow terminating. It forces that each connected component is either perfect or acyclic.}
\end{exm}

We recall from \cite[Lemma 3.4]{Chen2011} that for a semisimple abelian category $\mathcal{A}$ and an auto-equivalence $\Sigma$ on $\mathcal{A}$, there is a unique triangulated structure on $\mathcal{A}$ with $\Sigma$ the translation functor. Indeed, all the triangles are split. We denote the resulting triangulated category by $(\mathcal{A}, \Sigma)$. We call a triangulated category \emph{semisimple} provided that it is triangle equivalent to $(\mathcal{A}, \Sigma)$ for some semisimple abelian category $\mathcal{A}$.

Let $n\geq 1$. Consider the algebra automorphism $\sigma\colon k^n\rightarrow k^n$ defined by
\begin{align}\label{equ:sigma}
\sigma(\lambda_1, \lambda_2, \cdots, \lambda_n)=(\lambda_2, \cdots, \lambda_n, \lambda_1).
 \end{align}
 Then $\sigma$ induces an automorphism $\sigma^*$ on the category $k^n\mbox{-mod}$ by twisting the module actions. We denote by $\mathcal{T}_n=(k^n\mbox{-mod}, \sigma^*)$ the resulting triangulated category.

The main result in this section is as follows. It is inspired by the work \cite{Kal}, and extends \cite[Theorem 2.5(b)]{Kal}.

\begin{thm}\label{thm:stable}
Let  $A=kQ/I$ be a quadratic monomial algebra. Assume that $\mathcal{C}_1$, $\mathcal{C}_2, \cdots, \mathcal{C}_m$ are all the perfect components of $\mathcal{R}_A$, and that each $d_i$ denotes the number of vertices in $\mathcal{C}_i$. Then there is a triangle equivalence
$$A\mbox{-\underline{\rm Gproj}}\stackrel{\sim}\longrightarrow \mathcal{T}_{d_1}\times \mathcal{T}_{d_2}\times \cdots \times \mathcal{T}_{d_m}.$$
\end{thm}

\begin{proof}
Let $\alpha$ be a perfect arrow. In particular, the morphism $\theta_\alpha$ in (\ref{equ:0}) is an isomorphism; see Lemma \ref{lem:iso}. Let $\beta$ be a different arrow.  Then we have $\alpha A\cap A\beta=\alpha A\beta$. We apply Lemma \ref{lem:Hom} to infer that $\underline{\rm Hom}_A(A\alpha, A\beta)=0$.  We observe that $\alpha A\cap A\alpha= k\alpha\oplus \alpha A\alpha$. Applying  Lemma \ref{lem:Hom} again, we infer that $\underline{\rm Hom}_A(A\alpha, A\alpha)=k{\rm Id}_{A\alpha}$.

We recall from Proposition \ref{prop:tuilun} that up to isomorphism, all the indecomposable objects in $A\mbox{-\underline{\rm Gproj}}$ are of the form $A\alpha$, where $\alpha$ is a perfect arrow. Recall from Lemma \ref{lem:re-quiver}(3) that an arrow $\alpha$ is perfect if and only if the corresponding vertex in $\mathcal{R}_A$ belongs to a perfect component. From the above calculation on Hom-spaces, we deduce that the categories $A\mbox{-\underline{\rm Gproj}}$ and $\mathcal{T}_{d_1}\times \mathcal{T}_{d_2}\times \cdots \times \mathcal{T}_{d_m}$ are equivalent. In particular, both categories are semisimple abelian.  To complete the proof, it suffices to verify that such an equivalence respects the translation functors.

Recall that the translation functor $\Sigma$ on  $A\mbox{-\underline{\rm Gproj}}$ is a quasi-inverse of the syzygy functor $\Omega$. For a perfect arrow $\alpha$ lying in the perfect component $\mathcal{C}_i$,  its relation-cycle is of the form $\alpha=\alpha_1, \alpha_2, \cdots, \alpha_{d_i},\alpha_{d_i+1}=\alpha$. Recall from (\ref{equ:syzygy}) that $\Omega(A\alpha_i)\simeq A\alpha_{i-1}$, and thus $\Sigma (A\alpha_i)=A\alpha_{i+1}$. On the other hand, the translation functor on $\mathcal{T}_{d_i}$ is induced by the algebra automorphism $\sigma$ in (\ref{equ:sigma}). By comparing these two translation functors, we infer that they are respected by the equivalence.
\end{proof}

Let us illustrate the results by an example.

\begin{exm}
{\rm Let $Q$ be the following quiver
\[\xymatrix{ 1 \ar@/^/[r]^-\alpha & 2 \ar@/^/[l]^-{\beta} \ar@/^/[r]^-\gamma & 3. \ar@/^/[l]^-{\delta} }\]
Let $I$ be the two-sided ideal  of $kQ$ generated by $\{\beta\alpha, \alpha\beta, \delta\gamma\}$, and let $A=kQ/I$. Then the relation quiver $\mathcal{R}_A$ is as follows.
\[\xymatrix{ \alpha \ar@/^/[r]^-{[\beta\alpha]} & \beta \ar@/^/[l]^-{[\alpha\beta]} &  \gamma\ar[r]^-{[\delta\gamma]} & \delta }\]
By Proposition \ref{prop:charac}(1) the algebra $A$ is Gorenstein. We mention that $A$ is not a gentle algebra. Indeed, the algebra $A$ is $2$-Gorenstein. In particular, the bound of the selfinjective dimension obtained in Proposition \ref{prop:charac}(1) is not sharp.

By Lemma \ref{lem:re-quiver}(3), all the perfect paths in $A$ are $\{\alpha, \beta\}$. Hence, there are only two indecomposable non-projective Gorenstein-projective $A$-modules $A\alpha$ and $A\beta$. By Theorem \ref{thm:stable} we have a triangle equivalence $A\mbox{-\underline{Gproj}}\stackrel{\sim}\longrightarrow \mathcal{T}_2$.}
\end{exm}

We observe a slight extension of Theorem \ref{thm:stable}. For a monomial algebra $A=kQ/I$, an \emph{overlap} in $A$ is given by two perfect paths $p$ and $q$ which satisfy  one of the following conditions:
\begin{enumerate}
\item[(O1)] $p=q$, and $p=p'x$ and $q=xq'$ for some nontrivial paths $x$, $p'$ and $q'$ with the path $p'xq'$ nonzero.
\item[(O2)] $p\neq q$, and $p=p'x$ and $q=xq'$ for some nontrivial path $x$ with the path $p'xq'$ nonzero.
\end{enumerate}
We observe that if $A$ is quadratic monomial, there is no overlap in $A$ because all perfect paths are indeed arrows.

\begin{prop}\label{prop:overlap}
Let $A=kQ/I$ be a monomial algebra. We denote by $d_1, d_2, \cdots, d_m$ the lengths of all the relation-cycles in $A$, where we identify relation-cycles up to cyclic permutations. Then the following statements are equivalent.
\begin{enumerate}
\item There is no overlap in $A$.
\item There is a triangle equivalence $A\mbox{-\underline{\rm Gproj}}\stackrel{\sim}\longrightarrow \mathcal{T}_{d_1}\times \mathcal{T}_{d_2}\times \cdots \times \mathcal{T}_{d_m}.$
\item The stable category $A\mbox{-\underline{\rm Gproj}}$ is semisimple.
\end{enumerate}
\end{prop}

\begin{proof}
We make the following observation. The condition (O1) is equivalent to the condition that the inclusion  $kp\oplus pAp\subseteq pA\cap Ap$ is proper, by Lemma \ref{lem:Hom} which is equivalent to that the inclusion $k{\rm Id}_{Ap} \subseteq \underline{\rm Hom}_A(Ap, Ap)$ is proper. Similarly, the condition (O2) is equivalent to the condition that $\underline{\rm Hom}_A(Ap, Aq)\neq 0$.

Using the observation, the implication ``$(1)\Rightarrow (2)$"  follows by the same argument as in the proof of Theorem \ref{thm:stable}. The implication ``$(2)\Rightarrow (3)$" is trivial.

To show ``$(3)\Rightarrow (1)$", we assume that $A\mbox{-\underline{\rm Gproj}}$ is semisimple. Recall that for a perfect path $p$, $Ap$ is an indecomposable object in $A\mbox{-\underline{\rm Gproj}}$. Therefore, by the semisimplicity condition, we have that $\underline{\rm Hom}_A(Ap, Ap)$ is  a division algebra. On the other hand, we observe  that ${\rm End}_A(Ap)/{{\rm rad}\; {\rm End}_A(Ap)}$ is isomorphic to $k$, where  ${{\rm rad}\; {\rm End}_A(Ap)}$ denotes the Jacobson radical. It follows that $\underline{\rm Hom}_A(Ap, Ap)$ is isomorphic to $k$, that is, we have $\underline{\rm Hom}_A(Ap, Ap)=k{\rm Id}_{Ap}$.  For two distinct perfect paths $p$ and $q$, the indecomposable objects $Ap$ and $Aq$ are not  isomorphic. Therefore, by the semisimplicity condition, we have $\underline{\rm Hom}_A(Ap,  Aq)=0$.  Then by the observation above, we infer that $A$ has no overlap.
\end{proof}

\begin{exm}
{\rm Let $Q$ be the following quiver.
\[\xymatrix{
1 \ar @/^/[r]^-\alpha & 2\ar @/^/[l]^-\beta
}\]
Consider the ideal $I$ of $kQ$ generated by $\beta\alpha\beta\alpha$. Let $A=kQ/I$. Then all the perfect pairs are given by $(\beta, \alpha\beta\alpha)$, $(\beta\alpha, \beta\alpha)$ and $(\beta\alpha\beta, \alpha)$; compare Lemma \ref{lem:Nak}(3). It follows that the unique perfect path is $\beta\alpha$, whose relation-cycle has length one. Then there is no overlap in $A$. By Proposition \ref{prop:overlap}, we have a triangle equivalence $A\mbox{-\underline{\rm Gproj}}\stackrel{\sim}\longrightarrow \mathcal{T}_{1}$. We mention that this equivalence can be deduced from \cite[Proposition 1]{Ringel},  or from \cite[Corollary 3.11]{CY} because $A$ is $2$-Gorenstein. }
\end{exm}

\section{An Example: the Nakayama monomial case}

In this section, we describe another  example for  Theorem \ref{thm:main}, where the quiver is a basic cycle. In this case, the monomial algebra $A$ is Nakayama. We recover a key characterization result of Gorenstein-projective $A$-modules in \cite{Ringel}.

Let $n\geq 1$. Let $Z_n$ be a basic $n$-cycle with the vertex set $\{1, 2, \cdots, n\}$ and the arrow set $\{\alpha_1, \alpha_2,\cdots, \alpha_n\}$, where $s(\alpha_i)=i$ and $t(\alpha_i)=i+1$. Here, we identify $n+1$ with $1$. Indeed, the vertices are indexed by the cyclic group $\mathbb{Z}/n\mathbb{Z}$.  For each integer $m$, we denote by $[m]$ the unique integer satisfying  $1\leq [m]\leq n$ and $m\equiv [m]$ modulo $n$. Hence, for each vertex $i$, $t(\alpha_i)=[i+1]$.  We denote by $p_i^l$ the unique path in $Z_n$ starting at $i$ of length $l$. We observe that $t(p_i^l)=[i+l]$.

Let $I$ be a monomial ideal of $kZ_n$, and let $A=kZ_n/I$ be the corresponding monomial algebra. Then $A$ is a connected Nakayama algebra which is elementary and has no simple projective modules. Indeed, any connected Nakayama algebra which is elementary and has no simple projective modules is of this form.

For each $1\leq i\leq n$, we denote by $P_i=Ae_i$ the  indecomposable projective $A$-module corresponding to $i$. Set $c_i={\rm dim}_k\; P_i$. Following \cite{Gus}, we define a map $\theta\colon \{1, 2, \cdots, n\}\rightarrow \{1, 2, \cdots, n\}$ such that $\theta(i)=[i+c_i]$. An element in $\bigcap_{d\geq 0} {\rm Im}\; \theta^d$ is called \emph{$\theta$-cyclic}.  We observe that $\theta$ restricts to a permutation on the set of $\theta$-cyclic elements.

Following \cite{Ringel} we call the projective $A$-module $P_i$ \emph{minimal} if its radical ${\rm rad}\; P_i$ is non-projective, or equivalently, each nonzero proper submodule of $P_i$ is non-projective. Recall that the projective cover of ${\rm rad}\; P_i$ is $P_{[i+1]}$. Hence, the projective $A$-module $P_i$ is minimal if and only if $c_i\leq c_{[i+1]}$. We observe that if $P_i$ is non-minimal, we have $c_i=c_{[i+1]}+1$.

We denote by $S_i$ the simple $A$-module corresponding to $i$.  The following terminology is taken from \cite{Ringel}. If $P_i$ is minimal,  we will say that the vertex $i$, or the corresponding simple module $S_i$, is \emph{black}. The vertex $i$, or $S_i$, is \emph{$\theta$-cyclically black} if $i$ is $\theta$-cyclic and $\theta^d(i)$ is black for each $d\geq 0$.

We recall that $\mathbf{F}$ denotes the set of minimal paths contained in $I$.

\begin{lem}\label{lem:Nak}
Let $1\leq i\leq n$ and $l\geq 0$. Let $p, q$ be two nonzero nontrivial paths in $A$ such that $s(p)=t(q)$. Then we have the following statements.
\begin{enumerate}
\item The path $p_i^l$ belongs to $I$ if and only if $l\geq c_i$.
\item The path $p_i^l$ belongs to $\mathbf{F}$ if and only if the $A$-module $P_i$ is minimal and $l=c_i$.
\item  The pair $(p, q)$ is perfect if and only if the concatenation $pq$ lies in $\mathbf{F}$. In this case, the vertex $s(q)$ is black.
\item If $(p, q)$ is a perfect pair, then $t(p)=\theta(s(q))$.
\end{enumerate}
\end{lem}

\begin{proof}
Recall that $P_i=Ae_i$ has a basis given by $\{p_i^j\; |\; 0\leq j< c_i\}$. Then (1) follows trivially.

For the ``only if" part of (2), we assume that $p_i^l$ belongs to $\mathbf{F}$. By the minimality of $p_i^l$, we have $l=c_i$. Moreover, if $P_i$ is not minimal, we have $c_{[i+1]}=c_i-1$ and thus $p_{[i+1]}^{l-1}$ belongs to $I$. This contradicts to the minimality of $p_i^l$. By reversing the argument, we have the ``if" part.

The ``only if" part of (3) follows from Lemma \ref{lem:char}(${\rm P'1}$). Then $pq=p_i^l$ for $i=s(q)$ belongs to $\mathbf{F}$. By (2), the vertex $i$ is black. For the ``if" part, we apply Lemma \ref{lem:char}. We only verify ${\rm (P'2)}$. Assume that $pq'=\gamma \delta$ with $\delta\in \mathbf{F}$. In particular, the path $pq'$ lies in $I$ which shares the same terminating vertex with $pq$. By the minimality of $pq$, we infer that $q'$ is longer than $q$. By $t(q')=t(q)$, we infer that $q'=qx$ for some nonzero path $x$.

We observe by (3) and (2) that the length of $pq$ equals $c_i$. Hence, we have $t(p)=[s(q)+c_i]=\theta(s(q))$, which proves (4).
\end{proof}

The following result describes explicitly all the perfect paths in $A=kZ_n/I$. It is in spirit close to \cite[Lemma 5]{Ringel}.

\begin{prop}\label{prop:Nak}
Let $A=kZ_n/I$ be as above, and $p$ be a nonzero nontrivial path in $A$. Then the path $p$ is perfect if and only if both vertices $s(p)$ and $t(p)$ are $\theta$-cyclically black.
\end{prop}

\begin{proof}
For ``only if" part, we assume that the path $p$ is perfect. We take a relation-cycle $p=p_1, p_2, \cdots, p_m, p_{m+1}=p$. We apply Lemma \ref{lem:Nak}(3) and (4) to each perfect pair $(p_i, p_{i+1})$, and deduce that $s(p_{i+1})$ is black and $t(p_i)=\theta(s(p_{i+1}))$ is also black because of $t(p_i)=s(p_{i-1})$.   Moreover, we have $\theta(s(p_{i+1}))=s(p_{i-1})$, where the subindex is taken modulo $m$. Then each $s(p_i)$ is $\theta$-cyclic and so is $t(p_i)$. Indeed, they are all $\theta$-cyclically black.

For the ``if" part, we assume that both vertices $s(p)$ and $t(p)$ are $\theta$-cyclically black. We claim that there exists a perfect pair $(q, p)$ with both $s(q)$ and $t(q)$ $\theta$-cyclically black.

Since the vertex $i=s(p)$ is black, by Lemma \ref{lem:Nak}(2) the path $p_i^{c_i}$ belongs to $\mathbf{F}$. We observe that $p_i^{c_i}=qp$ for a unique nonzero path $q$. Then $(q, p)$ is a perfect pair by Lemma \ref{lem:Nak}(3). By  Lemma \ref{lem:Nak}(4), we have $t(q)=\theta(s(p))$. Hence, both vertices $s(q)=t(p)$ and $t(q)$ are $\theta$-cyclically black, proving the claim.

Set $q_0=p$ and $q_1=q$. We apply the claim repeatedly and obtain perfect pairs $(q_{i+1}, q_{i})$ for each $i\geq 0$. We assume that $q_l=q_{m+l}$ for some $l\geq 0$ and $m>0$. Then applying Lemma \ref{lem:unique} repeatedly, we infer that $q_0=q_m$. Then we have the desired relation-cycle $p=q_m, q_{m-1}, \cdots, q_1, q_0=p$ for the given path $p$.
\end{proof}

 As a consequence, we recover  a key characterization result of Gorenstein-projective $A$-modules in \cite[Lemma 5]{Ringel}. We denote by ${\rm top}\; X$ the top of an $A$-module $X$.

\begin{cor}
Let $A=kZ_n/I$ be as above, and  $M$ be an indecomposable non-projective $A$-module. Then the module $M$ is Gorenstein-projective if and only if both ${\rm top}\;M$ and ${\rm top}\; \Omega(M)$ are $\theta$-cyclically black simple modules.
\end{cor}

\begin{proof}
For the ``only if" part, we assume by Theorem \ref{thm:main} that $M=Ap$ for a perfect path $p$. We take a perfect pair $(q, p)$ with $q$ a perfect path. Then by (\ref{equ:3}) we have $\Omega(M)\simeq Aq$. We infer that ${\rm top}\;M\simeq S_{t(p)}$ and ${\rm top}\;\Omega(M)\simeq S_{t(q)}$. By Proposition \ref{prop:Nak}, both simple modules are $\theta$-cyclically black.

For the ``if" part, we assume that ${\rm top}\; M\simeq S_i$. Take a projective cover $\pi\colon P_i\rightarrow M$. Recall that each nonzero proper submodule of $P_i$ is of the form $Ap$ for a nonzero nontrivial path $p$ with $s(p)=i$. Take such a path $p$ with $Ap={\rm Ker}\; \pi$, which is isomorphic to $\Omega(M)$.  Therefore, by the assumption both $S_i=S_{s(p)}$ and $S_{t(p)}\simeq {\rm top}\; \Omega(M)$ are $\theta$-cyclically black. Then by Proposition \ref{prop:Nak}, the path $p$ is perfect. Take a perfect pair $(p, q)$ with $q$ a perfect path.  In particular, by (\ref{equ:3}) $Aq$ is isomorphic to $P_i/Ap$, which is further isomorphic to $M$. Then we are done, since by Proposition \ref{prop:bijection}(4) $Aq$ is a Gorenstein-projective module.
\end{proof}

\vskip 20pt

\noindent {\bf Acknowledgements} \;
X.W. Chen and D. Shen are supported by Wu Wen-Tsun Key Laboratory of Mathematics of CAS, the National Natural Science Foundation of China
(No.11201446), NECT-12-0507 and a grant from CAS.  G. Zhou is supported by Shanghai Pujiang Program (No.13PJ1402800),
 the National Natural Science Foundation of China (No.11301186) and the Doctoral Fund of Youth Scholars of Ministry of Education of China (No.20130076120001).

\bibliography{}

\begin{thebibliography}{999}





\bibitem{ABr} {\sc M. Auslander, and M. Bridger}, {\em Stable module
theory}, Mem. Amer. Math. Soc. {\bf 94}, 1969.






\bibitem{ARS}{\sc M. Auslander, I. Reiten, and S.O. Smal{\o},}
Representation Theory of Artin Algebras, Cambridge Studies in Adv.
Math. {\bf 36}, Cambridge Univ. Press, Cambridge, 1995.


\bibitem{Bel2} {\sc A. Beligiannis,} {\em Cohen-Macaulay modules, (co)torsion pairs and
virtually Gorenstein algebras}, J. Algebra {\bf 288} (2005),
137--211.






\bibitem{Buc} {\sc R.O. Buchweitz,} Maximal Cohen-Macaulay Modules
and Tate Cohomology over Gorenstein Rings, Unpublished Manuscript,
1987.




\bibitem{Chen2011} {\sc X.W. Chen}, {\em The singularity category of an algebra with radical square zero}, Doc. Math. {\bf 16} (2011), 921--936.




\bibitem{Chen2012} {\sc X.W. Chen}, {\em  Algebras with radical square zero are either self-injective or CM-free,}
Proc. Amer. Math. Soc. {\bf 140} (1) (2012), 93--98.


\bibitem{CY}{\sc X.W. Chen,  and Y. Ye}, {\em  Retractions and Gorenstein homological properties},
Algebr. Repre. Theor. {\bf 17} (2014), 713--733.




\bibitem{EJ1} {\sc E. Enochs, and O. Jenda,} {\em Gorenstein injective and projective
modules, } Math. Z. {\bf 220} (1995), 611--633.



\bibitem{Hap} {\sc D. Happel,} Triangulated Categories in the
Representation Theory of Finite Dimensional Algebras, London Math.
Soc., Lecture Notes Ser. {\bf 119}, Cambridge Univ. Press,
Cambridge, 1988.


\bibitem{Hap2} {\sc D. Happel,} {\em On Gorenstein algebras}, In:
Progress in Math. {\bf 95}, Birkh\"{a}user Verlag Basel, 1991,
389--404.


\bibitem{GR} {\sc Ch. Geiss,  and I. Reiten}, {\em Gentle algebras are Gorenstein}, in: Representations of algebras and related topics, 129--133,
Fields Inst. Commun.  {\bf 45}, Amer. Math. Soc., Providence, RI, 2005.


\bibitem{Gus} {\sc W.H. Gustafson}, {\em Global dimension in serial rings}, J. Algebra {\bf 97} (1985), 14--16.



\bibitem{HS} {\sc C. Holdaway, and P. Smith}, {\em  An equivalence of categories for graded modules over monomial algebras and path algebras of quivers,} J. Algebra {\bf 353} (2012) 249--260; Corrigendum, J. Algebra  {\bf 357} (2012) 319--321.



\bibitem{Hoshi} {\sc M. Hoshino,} {\em Algebras of finite self-injective dimension}, Proc. Amer. Math. Soc. {\bf 112}(3)(1991), 619--622.


\bibitem{Kal} {\sc M. Kalck}, {\em Singularity categories of gentle algebras}, Bull. London Math. Soc., to appear, arXiv:1207.6941v3.




\bibitem{Qui73} {\sc D. Quillen,} {\em Higher algebraical K-theory
I}, Springer Lecture Notes in Math. {\bf 341}, 1973, 85--147.


\bibitem{Ringel} {\sc C.M. Ringel}, {\em The Gorenstein projective modules for the Nakayama algebras. I.},
 J. Algebra {\bf 385} (2013), 241--261.



\bibitem{Zim} {\sc B. Zimmermann Huisgen}, {\em Predicting syzygies over monomial relation algebras}, Manu. Math. {\bf 70} (1991), 157--182.




\end{thebibliography}

\vskip 10pt

 {\footnotesize \noindent Xiao-Wu Chen, School of Mathematical Sciences, University of Science and Technology of
China, Hefei 230026, P.R. China \\
URL: http://home.ustc.edu.cn/$^\sim$xwchen\\

\footnotesize \noindent Dawei Shen, School of Mathematical Sciences, University of Science and Technology of
China, Hefei 230026, P.R. China \\
URL: http://home.ustc.edu.cn/$^\sim$sdw12345\\

\footnotesize \noindent Guodong Zhou, Department of Mathematics, Shanghai Key laboratory of PMMP, East China Normal  University, Shanghai
200241,  P.R. China}

\end{document}